\documentclass[10pt,a4paper]{article}
\usepackage{amsfonts}
\usepackage{amsthm}
\usepackage{amssymb}
\usepackage{mathrsfs}
\usepackage{geometry}
\usepackage{graphicx}
\usepackage[utf8]{inputenc}
\usepackage[T1]{fontenc}
\usepackage{amsmath}

\usepackage{authblk}

\usepackage{hyperref}
\hypersetup{backref, pdfpagemode=FullScreen, colorlinks=true,
	citecolor=magenta, linkcolor=cyan, urlcolor=blue}

\usepackage{mathtools}
\mathtoolsset{showonlyrefs}

\newcommand{\vol}{\mathrm{vol}\,}

\theoremstyle{change}
\newtheorem{theo}{Theorem}[section]
\newtheorem*{theo*}{Theorem}

\newtheorem{coro}[theo]{Corollary}
\newtheorem{prop}[theo]{Proposition}

\newtheorem{rem}[theo]{Remark}

\begin{document}
	\title{On Hadwiger's covering functional for the simplex and the cross-polytope}
	
	\author{Fei Xue\footnote{Fei Xue: 05429@njnu.edu.cn\\School of Mathematical Sciences, Nanjing Normal University, No.1 Wenyuan Road Qixia District, Nanjing, P.R.China 210046}, Yanlu Lian\footnote{Yanlu Lian: yanlu\_lian@tju.edu.cn\\Center for Applied Mathematics, Tianjin University, Tianjin, P.R.China 300354},  Yuqin Zhang\footnote{Yuqin Zhang: yuqinzhang@tju.edu.cn\\School of Mathematics, Tianjin University, Tianjin, P.R.China 300072}}
	
	\maketitle
	
	
	\begin{abstract}
		In 1957, Hadwiger made a conjecture that every $n$-dimensional convex body can be covered by $2^n$ translations of its interior. The Hadwiger's covering functional $\gamma_m(K)$ is the smallest positive number $r$ such that $K$ can be covered by $m$ translations of $rK$. Due to Zong's program, we study the Hadwiger's covering functional for the simplex and the cross-polytope. In this paper, we give upper bounds for the Hadwiger's covering functional of the simplex and the cross-polytope.
	\end{abstract}

\section{Introduction}

Let $\mathcal{K}^n$ be the set of all convex bodies, that is, compact convex sets in the $n$-dimensional Euclidean space $\mathbb{E}^n$ with non-empty interior, and let $K$ be a convex body with boundary $\partial{K}$, interior $\mathrm{int}(K)$. If $\boldsymbol{a_1}, \boldsymbol{a_2},\dots, \boldsymbol{a_n}$ are $n$ linearly independent vectors in $\mathbb{E}^n$, then we call
$$\Lambda=\{\sum_{i=1}^nz_i\boldsymbol{a_i}:z_i\in\mathbb{Z}\}$$
a lattice. The set $\{\boldsymbol{a_1}, \boldsymbol{a_2},\dots, \boldsymbol{a_n}\}$ is called a basis of the lattice. Let $X$ be a discrete set of points in $\mathbb{E}^n$. We call $K+X$ a translative covering of $\mathbb{E}^n$ if
$$\bigcup_{x\in X}(K+x)=\mathbb{E}^n.$$
In addition, if $X$ is a lattice, we will call $K+X$ a lattice covering of $\mathbb{E}^n$. Let $c(K)$ denote the covering number of $K,$ i.e., the smallest number of translations of $\mathrm{int}(K)$ such that their union can cover $K$. In 1955, Levi~\cite{Levi} studied the covering number in dimension $2$. He proved that
 $$c(K)=
 \begin{cases}
 4,&\text{if K is a parallelogram},\\
 3,&\text{otherwise.}
 \end{cases}$$
 In 1957, Hadwiger~\cite{Hadwiger} studied this number and proposed the famous conjecture:

{\bf{Hadwiger's covering conjecture:}} For every $K\in \mathcal{K}^n$ we have
 $$c(K)\leq2^n,$$
 where the equality holds if and only if $K$ is a parallelopiped.

 Lassak~\cite{Lassak1} proved this conjecture for the three-dimensional centrally symmetric case. Rogers and Zong~\cite{Rogers} obtained the currently best known upper bound
 $$c(K)\leq \binom{2n}{n}(n\log n+n\log\log n+5n)$$
 for general $n$-dimensional convex bodies, and
 $$c(K)\leq2^n(n\log n+n\log\log n+5n)$$
 for centrally symmetric ones. Combining ideas from~\cite{Artstein} with a new result on the K{\"o}vner-Besicovitch measure of symmetry for convex bodies, Huang et al.~\cite{Huang} obtained a new upper bound for the covering number. That is, there exist universal constants $c_1, c_2 > 0$ such that for all $n\geq 2$ and every convex body $K \in \mathcal{K}^n$,
 $$c(K)\leq c_14^ne^{c_2\sqrt{n}}.$$
 For more details about several connections with other important problems such as the illumination problem and the separation problem, we refer to~\cite{Artstein,Bezdek,Boltyanski2,Brass}. Nevertheless, we are still far away from the solution of the conjecture, even for the three-dimensional case.

Let $\gamma_m(K)$ denote the smallest number $\rho$ such that $K$ can be covered by $m$ translations of $\rho K$, i.e.,
  $$\gamma_m(K)=\min\{\lambda>0: \exists\{\boldsymbol u_i:i=1,\dots,m\}\subset \mathbb{E}^n\text{, such that}\quad K\subseteq \cup_{i=1}^{m}(\lambda K+\boldsymbol u_i)\},$$
and define
$$\gamma_n=\max\limits_{K}\{\gamma_{2^n}(K)\}.$$

For some special cases, Lassak~\cite{Lassak2} showed that for every two-dimensional convex domain $K$,
 $$\gamma_4(K)\leq \frac{\sqrt{2}}{2}.$$
In 1998, Lassak~\cite{Lassak3} showed that for every three-dimensional convex body $K$,
$$\gamma_{24}(K)\leq \frac{7}{12}\sqrt{2}+\frac{1}{6}=0.9916\cdots.$$

Zong~\cite{Zong2,Zong4} conjectured that
$$\lim_{n \to \infty}\gamma_n=1,$$
and he furtherly proposed a four-step program to attack them. In particular, let $S_n$ be an $n$-dimensional simplex, and let $C_n^\star$ be an $n$-dimensional cross-polytope. Zong asked whether
$$\lim_{n \to \infty}\gamma_{2^n}(S_n)~\&~\lim_{n \to \infty}\gamma_{2^n}(C_n^\star)=1.$$

In this paper, we are going to show:

\begin{theo}\label{theo:1}
Let $S_n=\{(x_1,x_2,\dots,x_n): \sum_{i=1}^nx_i\leq 1, x_i\geq0, i=1,\dots,n\}$ be a simplex, then we have $\lim_{n \to \infty}\gamma_{2^n}(S_n)\leq \frac{1}{1+c_1}$, where $c_1=0.2938\cdots$ is the solution of $\frac{(1+c)^{(1+c)}}{c^c}=2$.
\end{theo}

\begin{theo}\label{theo:2}
Let $C_n^\star=\{(x_1,x_2,\dots,x_n): \sum_{i=1}^n|x_i|\leq 1, i=1,\dots,n\}$ be a cross-polytope,
then we have $\lim_{n \to \infty}\gamma_{2^n}(S_n)\leq \frac{1}{1+c_2}$, where $c_3\leq c_2\leq c_4$, $c_3=0.2056\cdots$ is the solution of $\frac{1}{2^{1-c}}\cdot \frac{(1+c)^{(1+c)}}{c^c}=1$, and $c_4=0.2271\cdots$ is the solution of $\frac{1}{2^{1-c}}\cdot \frac{1}{(c^c)\cdot (1+c)^{(1+c)}}=1$.
\end{theo}

Let $K_n^{p}$ be the unit ball of the $n$-dimensional $l_p$ norm, in other words,
$$K_n^{p}=\{(x_1,x_2,\dots,x_n):\mid x_1\mid^p+\mid x_2\mid^p+\cdots+ \mid x_n \mid^p\leq 1\}.$$
Actually, when $p=1$, $K_n^{1}=C_n^\star$.
Let
$$K_n^{p*}=\{(x_1,x_2,\dots,x_n):x_i\geq 0, x_1^p+ x_2^p+\cdots+ x_n^p \leq 1\}.$$
Zong~\cite{Zong2} obtained
$$\gamma_8(C)\leq \frac{2}{3}$$
for a bounded three-dimensional convex cone $C$, and
$$\gamma_8(K_3^{p})\leq \sqrt{\frac{2}{3}}$$
for all the unit ball $K_3^{p}$ of the three-dimensional $l_p$ spaces.
Lian and Zhang~\cite{Lian} obtained a series of exact values of $\gamma_m(K_3^{1})$, when $m$ is some positive integer numbers.
In this paper, we will furtherly show that:

\begin{theo}\label{theo:3}
For each $p\geq 1$, $\lim_{n \to \infty}\gamma_{2^n}(K_n^{p*})\leq \left(\frac{1}{1+c_1}\right)^{\frac{1}{p}}$, where $c_1=0.2938\cdots$ is the solution $\frac{(1+c)^{1+c}}{c^c}=2.$	
\end{theo}

\begin{theo}\label{theo:4}
For each $p\geq 1$, $\lim_{n \to \infty}\gamma_{2^n}(K_n^{p})\leq \left(\frac{1}{1+c_2}\right)^{\frac{1}{p}}$, where $c_3\leq c_2\leq c_4$. And $c_3=0.2056\cdots$ is the solution of $\frac{1}{2^{1-c}}\cdot \frac{(1+c)^{(1+c)}}{c^c}=1$, $c_4=0.2271\cdots$ is the solution of $\frac{1}{2^{1-c}}\cdot \frac{1}{(c^c)\cdot (1+c)^{(1+c)}}=1$.
\end{theo}

This paper is organized as follows. In Section 2, we provide a fundamental idea for finite coverings of the simplex and the cross-polytope, and give a short introduction to the lattice covering density. In Section 3, we deal with the Hadwiger's covering functionals of the simplex and the cross-polytope by an easiest lattice covering method. In Section 4, we generalize the finite coverings of the simplex and the cross-polytope to the quarter-$l_p$ ball and $l_p$ ball. Within the same way, we give upper bounds of Hadwiger's covering functionals of the quarter-$l_p$ ball and $l_p$ ball.

\section{The easiest lattice coverings of the simplex and the cross-polytope}

For convenience, we let $\overline{S_n}=\{(x_1,x_2,\dots,x_n): \sum_{i=1}^nx_i\leq n,x_i\geq0,i=1,\dots,n\}$ be a normalized simplex, and $\overline{C_n^\star}=\{(x_1,x_2,\dots,x_n):\mid x_1\mid+\mid x_2\mid+\dots+ \mid x_n \mid\leq n\}$ be a normalized cross-polytope. Both lattice coverings are not optimal, i.e., the corresponding covering densities of $\mathbb{E}^n$ are not the best. For more details we refer to Section~\ref{sect:remark}.
However, they provide a fundamental idea for finite coverings of the simplex and the cross-polytope.

\subsection{Simplex covering}

Let $M_1(n,k)=\{(x_1,x_2,\dots,x_n)\in \mathbb{Z}^n:\sum_{i=1}^nx_i\leq k, x_i\geq0,i=1,\dots,n\}.$

\begin{prop}\label{prop:simplex}
	$\overline{S_n}+M_1(n,k)=\frac{n+k}{n}\overline{S_n}$. Moreover,
	$$\#M_1(n,k)=\binom{n+k}{n}.$$
\end{prop}

\begin{proof}
Since $\overline{S_n}=\{(x_1,x_2,\dots,x_n): \sum_{i=1}^nx_i\leq n,x_i\geq0,i=1,\dots,n\}$, then,
$$\frac{n+k}{n}\overline{S_n}=\{(x_1,x_2,\dots,x_n): \sum_{i=1}^nx_i\leq n+k, x_i\geq0,i=1,\dots,n\}.$$
It is easy to see $\overline{S_n}+M_1(n,k)\subseteq \frac{n+k}{n}\overline{S_n}$.
On the other side, for $k=0$, $\overline{S_n}+M_1(n,0)=\overline{S_n}$. By induction, we assume that
$$\overline{S_n}+M_1(n,k-1)\supseteq \frac{n+k-1}{n}\overline{S_n}.$$
Then for any $(y_1,y_2,\dots,y_n)\in \frac{n+k}{n}\overline{S_n}$, we may assume that
$$k+n-1<y_1+y_2+ \dots +y_n\leq k+n, \eqno(1a)$$
otherwise it will be contained in $\frac{n+k-1}{n}\overline{S_n}.$ Now let
$$y_i=[y_i]+\{y_i\},\eqno(1b)$$
where $[\cdot]$ is the Gauss function, and $0\leq\{y_i\}<1$ for all $i=1,2,\dots,n.$ That is,
$$k+n-1<[y_1]+\{y_1\}+[y_2]+\{y_2\}+\dots+[y_n]+\{y_n\}\leq k+n.$$
Apparently, we have $\{y_1\}+\{y_2\}+\dots+\{y_n\}\leq n,$ thus
$$k-1<[y_1]+[y_2]+\cdots+[y_n].$$
Since $[y_1]+[y_2]+\cdots+[y_n]$ is an integer, we have
$$k\leq[y_1]+[y_2]+\cdots+[y_n].$$
Consider integers $0\leq z_i\leq [y_i]$, $i=1,\cdots,n$, with $k=z_1+\cdots+z_n$. In this case,
$$(y_1,\cdots,y_n)=(z_1,\cdots,z_n)+(y_1-z_1,\cdots,y_n-z_n),$$
where $(z_1,\cdots,z_n)\in M_1(n,k)$, $y_i-z_i\geq 0$, and $\sum_{i=1}^{n}(y_i-z_i)\leq n$.

Therefore,
 $$\frac{n+k}{n}\overline{S_n}=\overline{S_n}+M_1(n,k).$$

Moreover,
\begin{align*}
\#M_1(n,k):=m_1(n,k)&=\#\{(x_1,x_2,\dots,x_n)\in \mathbb{Z}^n:x_i\geq0, \sum_{i=1}^n x_i\leq k\}\\
&=\#\{(x_1,x_2,\dots,x_n)\in \mathbb{Z}^n:x_i\geq0, \sum_{i=1}^n x_i+t=k, t\geq 0, t\in \mathbb{Z}\}\\
&=\binom{n+k}{n}.
\end{align*}
\end{proof}

\begin{coro}
$\gamma_{m_1(n,k)}(S_n)\leq \frac{n}{n+k}$.
\end{coro}

\subsection{Cross-polytope covering}

Let $M_2(n,k)=\{(x_1,x_2,\dots,x_n)\in \mathbb{Z}^n:\sum_{i=1}^n\mid x_i\mid\leq k, i=1,\dots,n\}.$

\begin{prop}\label{prop:cross}
	$\overline{C_n^\star}+M_2(n,k)=\frac{n+k}{n}\overline{C_n^\star}.$ Moreover,
	$$\#M_2(n,k)=\sum_{j=0}^n\binom{n}{j}\binom{k+j}{n}.$$
\end{prop}

\begin{proof}
This proof is similar with the proof of Proposition \ref{prop:simplex}. First, we have
$$\frac{n+k}{n}\overline{C_n^\star}=\{(x_1,x_2,\dots,x_n):\sum_{i=1}^n\mid x_i\mid\leq n+k,i=1,\dots,n\}.$$
On one side, it is easy to see $\overline{C_n^\star}+M_2(n,k)\subseteq\frac{n+k}{n}\overline{C_n^\star}.$
On the other side, for $k=0$, $\overline{C_n^\star}+M_2(n,0)=\overline{C_n^\star}$. By induction, we assume that
$$\overline{C_n^\star}+M_2(n,k-1)\supseteq \frac{n+k-1}{n}\overline{C_n^\star}.$$
Then for any $(y_1,y_2,\dots,y_n)\in \frac{n+k}{n}\overline{C_n^\star}$, we may assume that $y_i\geq 0$ and
$$k+n-1<y_1+y_2+ \dots +y_n\leq k+n, \eqno(2a)$$
otherwise it will be contained in $\frac{n+k-1}{n}\overline{C_n^\star}.$ Now let
$$y_i=[y_i]+\{y_i\},\eqno(2b)$$
where $[\cdot]$ is the Gauss function, and $0\leq\{y_i\}<1$ for all $i=1,2,\dots,n.$ That is,
$$k+n-1<[y_1]+\{y_1\}+[y_2]+\{y_2\}+\dots+[y_n]+\{y_n\}\leq k+n.$$
Apparently, we have $\{y_1\}+\{y_2\}+\dots+\{y_n\}\leq n,$ thus
$$k-1<[y_1]+[y_2]+\cdots+[y_n].$$
Since $[y_1]+[y_2]+\cdots+[y_n]$ is an integer, we have
$$k\leq[y_1]+[y_2]+\cdots+[y_n].$$
Consider integers $0\leq z_i\leq [y_i]$, $i=1,\cdots,n$, with $k=z_1+\cdots+z_n$. In this case,
$$(y_1,\cdots,y_n)=(z_1,\cdots,z_n)+(y_1-z_1,\cdots,y_n-z_n),$$
where $(z_1,\cdots,z_n)\in M_1(n,k)$, $y_i-z_i\geq 0$, and $\sum_{i=1}^{n}(y_i-z_i)\leq n$.

Therefore, we have
$$\frac{n+k}{n}\overline{C_n^\star}=\overline{C_n^\star}+M_2(n,k).$$

And moreover,
\begin{align*}
\#\{M_2(n,k)\}:=m_2(n,k)&=\#\{(x_1,x_2,\dots,x_n)\in \mathbb{Z}^n:\sum_{i=1}^n \mid x_i\mid\leq k\}\\
&=\sum_{j=0}^k\#\{(x_1,x_2,\dots,x_n)\in \mathbb{Z}^n:\sum_{i=1}^{n-1}\mid x_i\mid\leq k-j, \mid x_n\mid=j\}\\
&=m_2(n-1,k)+2\sum_{j=0}^{k-1}m_2(n-1,j).
\end{align*}
Therefore,
$$m_2(n,k)=m_2(n-1,k-1)+m_2(n-1,k)+m_2(n,k).$$
In order to show
$$m_2(n,k)=\sum_{j=0}^n\binom{n}{j}\binom{k+j}{n}=\sum_{-\infty}^{\infty}\binom{n}{j}\binom{k+j}{n},$$
we use induction on $n$ and $k$.

It is obvious that $m_2(1,k)=2k+1$, $m_2(2,k)=2k^2+2k+1$.
If $m_2(n-1,k)=\sum_{-\infty}^{\infty}\binom{n-1}{j}\binom{k+j}{n-1}$, and if $m_2(n,k-1)=\sum_{-\infty}^{\infty}\binom{n}{j}\binom{k-1+j}{n}$, then
\begin{align*}
m_2(n,k)&=\sum_{-\infty}^{\infty}\binom{n-1}{j}\binom{k-1+j}{n-1}+\sum_{-\infty}^{\infty}\binom{n-1}{j}\binom{k+j}{n-1}+\sum_{-\infty}^{\infty}\binom{n}{j}\binom{k-1+j}{n}\\
&=\sum_{-\infty}^{\infty}\binom{n-1}{j+1}\binom{k+j}{n-1}+\sum_{-\infty}^{\infty}\binom{n-1}{j}\binom{k+j}{n-1}+\sum_{-\infty}^{\infty}\binom{n}{j}\binom{k-1+j}{n}\\
&=\sum_{-\infty}^{\infty}\binom{n}{j+1}\binom{k+j}{n-1}+\sum_{-\infty}^{\infty}\binom{n}{j}\binom{k-1+j}{n}\\
&=\sum_{-\infty}^{\infty}\binom{n}{j}\binom{k-1+j}{n-1}+\sum_{-\infty}^{\infty}\binom{n}{j}\binom{k-1+j}{n}\\
&=\sum_{-\infty}^{\infty}\binom{n}{j}\binom{k+j}{n}.
\end{align*}
\end{proof}

\begin{coro}
$\gamma_{m_2(n,k)}(C_n^\star)\leq \frac{n}{n+k}$.
\end{coro}

\subsection{Remark}\label{sect:remark}
For any convex body $K\in \mathcal{K},$ if $\mathcal{K}=\{K_i: K_i\quad\text{are congruent to K}\}$ is a covering of $\mathbb{E}^n$,
we define translative covering density
$$\theta(\mathcal{K})=\lim_{\ell \rightarrow \infty}\inf\frac{\vol(\mathcal{K}\cap \ell W_n)}{\vol(\ell W_n)},$$
where $W_n$ denote the $n$-dimensional unit cube. The lattice covering density of $K$ is
$$\theta^l(K)=\min_{\mathcal{K}}\{\theta(\mathcal{K}): \mathcal{K}\quad\text{is a lattice covering}\}.$$
It is mentioned by Rogers and Zong~\cite{Rogers} that when $n\geq 3$, $r\in (0,1)$ and $K$ is centrally symmetric, then
$$N(K,rK)\leq (1+r^{-1})^n(n\log{n}+\log\log{n}+5n),$$
where $N(K,rK)$ (or $N^l(K,rK)$) is the smallest number of (lattice) translations of $rK$ required to cover $K$.
Equivalently,
$$\gamma_{(1+r^{-1})^n (n\log{n}+\log\log{n}+5n)}(K)\leq r.$$
All the estimates follow immediately from two simple general results for
general convex bodies $H$ and $K$, i.e.,
$$N(K,H)\leq \frac{\vol(K-H)}{\vol(H)}\theta(H),$$
and
$$N^l(K,H)\leq \frac{\vol(K-H)}{\vol(H)}\theta^l(H),$$
where $\vol(\cdot)$ denotes the volume of a set, and $K-H$ denotes the Minkowski difference
of $K$ and $H$.
That is, they provided the upper bound via the most economical covering of $K$. Here we only use a very easy covering of $S_n$ and $C_n^\star.$
\section{Hadwiger's covering functional for the simplex}

In this section, we prove Theorem \ref{theo:1} based on Proposition \ref{prop:simplex}.

\begin{proof}[Proof of Theorem \ref{theo:1}]
Let $c_1$ be the solution of $\frac{(1+c)^{1+c}}{c^c}=2$, i.e., $c_1=0.2938\cdots$. Then by the Stirling formula,

\begin{align*}
\lim_{n\to+\infty}\left(\binom{n+[c_1n]}{n}\right)^{\frac{1}{n}}&=\lim_{n\to+\infty} \left(\frac{(n+[c_1n])!}{n![c_1n]!}\right)^{\frac{1}{n}}\\
&=\lim_{n\to+\infty} \left(\frac{(n+[c_1n])^{n+[c_1n]}\cdot \sqrt{2\pi(n+[c_1n])}}{n^n[c_1n]^{[c_1n]}\cdot\sqrt{2\pi n}\sqrt{2\pi [c_1n]}}\right)^{\frac{1}{n}}\\
&=\lim_{n\to+\infty}\left(\frac{(1+c_1)^{(1+c_1)}}{{c_1}^{c_1}}\right)=2
\end{align*}
Therefore, if we let $\binom{n+k(n)}{n}\leq 2^n< \binom{n+k(n)+1}{n}$, then we have $\lim_{n\to \infty}\frac{k(n)}{n}=c_1$.
Thus,
$$\lim_{n\to \infty}\gamma_{2^n}(S_n)\leq \frac{1}{1+c_1}.$$
\end{proof}
From the same reason, we have:
\begin{coro}
If $a(t)$ is the solution of $t=\frac{(1+x)^{(1+x)}}{x^x}$, for $t>1$, then $\lim_{n \to +\infty}\gamma_{t^n}(S_n)\leq \frac{1}{1+a(t)}.$
\end{coro}

\section{Hadwiger's covering functional for the cross-polytope}

In this section, we prove Theorem \ref{theo:2} based on Proposition \ref{prop:cross}.

\begin{proof}[Proof of Theorem \ref{theo:2}]
We observe that $M_2(n,k)=M_2(k,n).$ Since
$$2^k\cdot \binom{n}{k}\leq \sum_{j=0}^k\binom{n}{j}\binom{k+j}{n}\leq 2^k\cdot \binom{n+k}{k},$$
we let $c_3$ be the solution of $\frac{1}{2^{1-c}}\cdot \frac{(1+c)^{(1+c)}}{c^c}=1$, i.e., $c_3=0.2056\cdots$, and let $c_4$ be the solution of $\frac{1}{2^{1-c}}\cdot \frac{1}{(c^c)\cdot (1+c)^{(1+c)}}=1$, i.e., $c_4=0.2271\cdots.$
Then by the Stirling formula,
\begin{align*}
\lim_{n\to +\infty}\left(2^{[c_3n]}\cdot\binom{n+[c_3n]}{[c_3n]}\right)^{\frac{1}{n}}&=\lim_{n\to +\infty}\left(2^{[c_3n]}\cdot\frac{(n+[c_3n])!}{n![c_3n]!} \right)^{\frac{1}{n}}\\
&=\lim_{n\to +\infty}\left(2^{[c_3n]}\frac{(n+[c_3n])^{n+[c_3n]}\cdot \sqrt{2\pi(n+[c_3n])}}{n^n[c_3n]^{[c_3n]}\cdot\sqrt{2\pi n}\sqrt{2\pi [c_3n]}}\right)^{\frac{1}{n}}\\
&=\lim_{n\to+\infty}\left(2^{c_3}\frac{(1+c_3)^{(1+c_3)}}{{c_3}^{c_3}}\right)=2
\end{align*}

\begin{align*}
\lim_{n\to +\infty}\left(2^{[c_4n]}\cdot\binom{n}{[c_3n]}\right)^{\frac{1}{n}}&=\lim_{n\to +\infty}\left(2^{[c_4n]}\cdot\frac{(n)!}{(n-[c_4n])![c_4n]!} \right)^{\frac{1}{n}}\\
&=\lim_{n\to +\infty}\left(2^{[c_4n]}\frac{n^n\cdot \sqrt{2\pi n}}{(n-[c_4n])^{(n-[c_4n])}[c_4n]^{[c_4n]}\cdot\sqrt{2\pi (n-[c_4n])}\sqrt{2\pi [c_4n]}}\right)^{\frac{1}{n}}\\
&=\lim_{n\to+\infty}\left(2^{c_4}\frac{1}{({c_4}^{c_4})\cdot(1+c_4)^{(1+c_4)}}\right)=2
\end{align*}
Therefore, let $2^{k_1(n)}\binom{n+k_1(n)}{k_1(n)}\leq 2^n<2^{k_1(n)+1}\binom{n+k_1(n)+1}{k_1(n)+1}$ and $2^{k_2(n)}\binom{n}{k_2(n)}\leq 2^n<2^{k_2(n)+1}\binom{n}{k_2(n)+1}$, we have $\lim_{n\to +\infty}\frac{k_1(n)}{n}=c_3$ and $\lim_{n\to +\infty}\frac{k_2(n)}{n}=c_4$. Thus, there exists $c_3\leq c_2\leq c_4$, such that
$$\lim_{n\to +\infty}\gamma_{2^n}(C_n^\star)\leq \frac{1}{1+c_2}.$$

\end{proof}

\section{From the simplex to the quarter-$l_p$ ball}
For convenience, we let $\overline{K_n^{p*}}=\{(x_1,x_2,\dots,x_n): \sum_{i=1}^n x_i^p\leq n,x_i\geq0,i=1,\dots,n\}$ be a normalized quarter-$l_p$ ball,
Similar with the easiest covering of the simplex, we construct a finite covering of the quarter-$l_p$ ball.
\begin{prop}\label{prop:simplex2}
For each $p\geq1$, $\gamma_{\binom{n+k}{n}}(\overline{K_n^{p*}})\leq \left(\frac{n}{n+k}\right)^{1/p}.$
\end{prop}

\begin{proof}
Similar with the easiest  covering of the simplex, we consider the covering $\overline{K_n^{p*}}+M_1(n,k)$ where $\#M_1(n,k)=\binom{n+k}{n}$ (or $=\binom{n+k}{k})$.

{\bf Claim:} $\overline{K_n^{p*}}+M_1(n,k)\supseteq t_{n,p,k}\overline{K_n^{p*}}$, where $t_{n,p,k}$ is inductively defined by $(t_{n,p,k+1}-1)^p+(n-1)t_{n,p,k+1}^p=nt_{n,p,k}^p.$

{\bf Proof of the Claim:} For $k=0$, $t_{n,p,0}=1$, i.e., $\overline{K_n^{p*}}+M_1(n,0)\supseteq t_{n,p,0}\overline{K_n^{p*}}$.

By induction, we assume that
$$\overline{K_n^{p*}}+M_1(n,k)\supseteq t_{n,p,k}\overline{K_n^{p*}},$$
then
\begin{align*}
\overline{K_n^{p*}}+M_1(n,k+1)&=\overline{K_n^{p*}}+M_1(n,k)+M_1(n,1)\\
&\supseteq t_{n,p,k}\overline{K_n^{p*}}+M_1(n,1).
\end{align*}
For each $x=(x_1,x_2,\dots, x_n)$ satisfying $x_i\geq 0$ and $\sum_{i=1}^nx_i^p=nt_{n,p,k+1}^p$, i.e., $x\in t_{n,p,k+1}\overline{K_n^{p*}},$
without loss of generality, we may assume that $x_1\geq x_2\geq \cdots\geq x_n.$ Apparently, $x_1\geq t_{n,p,k+1}\geq 1$. Due to the monotonicity of $(t-1)^p-t^p$, we have
\begin{align*}
(x_1-1)^p+\cdots+x_n^p&=x_1^p+\cdots+x_n^p+(x_1-1)^p-x_1^p\\
&\leq nt_{n,p,k+1}^p+\left(t_{n,p,k+1}-1\right)^p-t_{n,p,k+1}^p\\
&=\left(t_{n,p,k+1}-1\right)^p+(n-1)t_{n,p,k+1}^p\\
&=nt_{n,p,k}^p.
\end{align*}
Therefore, $x\in t_{n,p,k}^p\overline{K_n^{p*}}+(1,0,\dots,0)\subset t_{n,p,k}^p\overline{K_n^{p*}}+M_1(n,1)$. Thus,
\begin{align*}
\overline{K_n^{p*}}+M_1(n,k+1)&\supseteq t_{n,p,k}\overline{K_n^{p*}}+M_1(n,1)\\
&\supseteq t_{n,p,k+1}\overline{K_n^{p*}}.
\end{align*}

Now we continue. Since $t^p-(t-1)^p$ is increasing in $(1, +\infty)$, we have
\begin{align*}
t_{n,p,k+1}^p-t_{n,p,k}^p&=\frac{1}{n}\left(t_{n,p,k+1}^p-\left(t_{n,p,k+1}-1\right)^p\right)\\
&\geq \frac{1}{n}.
\end{align*}
Combining with $t_{n,p,1}^p-1\geq \frac{1}{n}$, we have
$$t_{n,p,k}^p\geq \frac{n+k}{n},$$
That is, $\overline{K_n^{p*}}+M_1(n,k)\supseteq \left(\frac{n+k}{n}\right)^{\frac{1}{p}}\overline{K_n^{p*}}.$

\end{proof}

\begin{proof}[Proof of Theorem \ref{theo:3}]
Similar with the Proof of Theorem \ref{theo:1}, we let $c_1$ be the solution of $\frac{(1+c)^{1+c}}{c^c}=2$, i.e., $c_1=0.2938\cdots$.
Then by the Stirling formula, we have
$$\lim_{n\to +\infty}\binom{n+[c_1n]}{n}^{\frac{1}{n}}=\lim_{n\to+\infty}\left(\frac{(1+c_1)^{(1+c_1)}}{{c_1}^{c_1}}\right)=2.$$
Therefore, let $\binom{n+k(n)}{n}\leq 2^n< \binom{n+k(n)+1}{n}$. Then we have $\lim_{n\to \infty}\frac{k(n)}{n}=c_1$.
Thus,
$$\lim_{n \to \infty}\gamma_{2^n}(K_n^{p*})\leq \left(\frac{1}{1+c_1}\right)^{\frac{1}{p}}.$$
\end{proof}

\section{From the cross-polytope to the $l_p$ ball}

Similar with the easiest covering of the cross-polytope, we construct a finite covering of the $l_p$ ball.

\begin{prop}\label{prop:cross2}
For each $p\geq1$, $\gamma_{\sum_{j=0}^n\binom{n}{j}\binom{k+j}{n}}K_n^{p}\leq \left(\frac{n}{n+k}\right)^{1/p}.$
\end{prop}

\begin{proof}
Similar with the crosspolytope, we consider the covering $K_n^{p}+M_2(n,k)$ where $\#M_2(n,k)=\sum_{j=0}^k\binom{n}{j}\binom{k+j}{n}$. Similar with the proof of Proposition \ref{prop:simplex2}, we have
$$\overline{K_n^p}+M_2(n,k)\supseteq \left(\frac{n+k}{n}\right)^{1/p}\overline{K_n^p}.$$
This implies that $\gamma_{\sum_{j=0}^n\binom{n}{j}\binom{k+j}{n}}K_n^{p}\leq \left(\frac{n}{n+k}\right)^{1/p}.$
\end{proof}

\begin{proof}[Proof of Theorem \ref{theo:4}]
Similar with the proof of Theorem \ref{theo:2}, let $c_3$ be the solution of $\frac{1}{2^{1-c}}\cdot \frac{(1+c)^{(1+c)}}{c^c}=1$, and let $c_4$ be the solution of $\frac{1}{2^{1-c}}\cdot \frac{1}{(c^c)\cdot (1+c)^{(1+c)}}=1.$
Then for the same reason, by the Stirling formula, we have
$$\lim_{n\to +\infty}\left(2^{[c_3n]}\cdot\binom{n+[c_3n]}{[c_3n]}\right)^{\frac{1}{n}}=\lim_{n\to+\infty}\left(2^{c_3}\frac{(1+c_3)^{(1+c_3)}}{{c_3}^{c_3}}\right)=2,$$
$$\lim_{n\to +\infty}\left(2^{[c_4n]}\cdot\binom{n}{[c_3n]}\right)^{\frac{1}{n}}=\lim_{n\to+\infty}\left(2^{c_3}\frac{(1+c_3)^{(1+c_3)}}{{c_3}^{c_3}}\right)=2.$$
Therefore, let $2^{k_1(n)}\binom{n+k_1(n)}{k_1(n)}\leq 2^n<2^{k_1(n)+1}\binom{n+k_1(n)+1}{k_1(n)+1}$ and $2^{k_2(n)}\binom{n}{k_2(n)}\leq 2^n<2^{k_2(n)+1}\binom{n}{k_2(n)+1}$. Then we have $\lim_{n\to +\infty}\frac{k_1(n)}{n}=c_3$, $\lim_{n\to +\infty}\frac{k_2(n)}{n}=c_4$. Thus, there exists $c_3\leq c_2\leq c_4$, such that
$$\lim_{n\to +\infty}\gamma_{2^n}(K_n^p)=\left(\frac{1}{1+c_2}\right)^{\frac{1}{p}}.$$

\end{proof}

\begin{rem}
For $p=2$, $K_n^2$ is the unit ball of $\mathbb{E}^n$, and we have $\lim_{n\to +\infty}\gamma_{2^n}(K_n^2)\leq\left(\frac{1}{1+c_2}\right)^{\frac{1}{2}},$ where $0.2056\cdots\leq c_2\leq 0.2271\cdots.$
\end{rem}

\section*{Acknowledgment}

The second author and the third author are supported by the National Nature Science Foundation of China
(NSFC 11921001) and the National Key Research and Development Program of China (2018YFA0704701).

\end{document}